\documentclass[11pt]{article}
\usepackage{enumitem}

\usepackage{lineno}

\usepackage{amsfonts,amsmath,amsthm,amscd,amssymb,latexsym,amsbsy,pb-diagram,cite,caption,url,hyperref}

\usepackage{tikz}
\usetikzlibrary{positioning, calc, trees, decorations.markings, patterns, arrows, arrows.meta}
\usetikzlibrary{trees}

\usepackage{diagmac2}
\usepackage{float}
\usepackage{circledtext}

\newtheorem{theorem}{Theorem}[section]
\newtheorem{lemma}[theorem]{Lemma}
\newtheorem{corollary}[theorem]{Corollary}

\theoremstyle{definition}
\newtheorem{definition}[theorem]{Definition}

\theoremstyle{remark}

\newtheorem{remark}[theorem]{Remark}

\newcommand{\keywords}{\textbf{Key words. }\medskip}
\newcommand{\subjclass}{\textbf{MSC 2020. }\medskip}
\renewcommand{\abstract}{\textbf{Abstract. }\medskip}

\numberwithin{equation}{section}

\hypersetup{breaklinks=true,colorlinks=true, linkcolor=black,filecolor=black,citecolor = black,  urlcolor=blue,}

\begin{document}

\author{Oleksiy Dovgoshey, Olga Rovenska}

\title{\bf Longest paths in trees and isometricity of ultrametric spaces}

\maketitle

\begin{abstract}
Let $T$ be a tree of arbitrary finite or infinite order and let $U(T)$ be the set of all ultrametric spaces generated by vertex labelings of $T$. Let ${\bf US}$ denote the class of all ultrametric spaces generated by vertex labelings of star graphs. We prove that the inclusion $U(T)\subseteq {\bf US}$ holds if and only if the longest path in $T$ has a length not exceeding three. 
\end{abstract}



\section{Introduction}

The concept of  ultrametric spaces generated by  non-negative vertex labelings on both finite and infinite trees was introduced in \cite{Dov2020TaAoG} and investigated in \cite{DK2024DLPSSAG,DK2022LTGCCaDUS,DR2025Arxx,DV2025JMS}. 
The ultrametric spaces generated by vertex labeled star graphs are studied in \cite{DCR2025OJAC,DR2025USGbLSG,DV2025JMS,DR2025Arxx2}.

The main goal of the present paper is to describe, up to isomorphism, the structure of trees $T$ which satisfy the following condition: If $d$ is an arbitrary ultrametric generated by vertex labeling of $T$, then the ultrametc space $(V(T),d)$ is isometric to an ultrametric space generated by vertex labeling of some star graph. The structure of such trees is described in Theorem~\ref{sun}, and is the main result of the paper. 
Corollary~\ref{xk} of this theorem gives a new interconnection between star graphs and double-star graphs of arbitrary order.

\section{Preliminaries.
Metric spaces and graphs}

\hspace{5 mm} Let us denote by $\mathbb{R}^+$ the set $[0, \infty)$.

An \textit{ultrametric} on a nonempty set $X$ is a function $d\colon X\times X\rightarrow \mathbb R^+$ satisfying the following conditions for all 
 \(x\), \(y\), \(z \in X\):
\begin{itemize}[left=10pt]
\item[(i)]  $d(x,y)=d(y,x)$, the symmetry property;
\item[(ii)] $(d(x,y)=0)\iff (x=y)$, the positivity property;
\item[(iii)] 
\(d(x,y) \leq \max\{d(x,z), d(z,y)\}\), the strong triangle inequality.
\end{itemize}

\begin{definition}\label{d2.2}
Let \((X, d)\) and \((Y, \rho)\) be ultrametric spaces. A bijective mapping \(\Phi \colon X \to Y\) is called an {\it isometry} of \((X, d)\) and \((Y, \rho)\) if
\[
d(x,y) = \rho(\Phi(x), \Phi(y))
\]
holds for all \(x\), \(y \in X\).  The ultrametric spaces are said to be  {\it isometric} if there is an isometry of these spaces.
\end{definition}

 A \textit{simple graph} is a pair $(V, E)$, where $V$ is a nonempty set and $E$ is a set of two-point subsets of $V$. For a simple graph $G = (V, E)$, the sets $V = V(G)$ and $E = E(G)$ are called the \textit{vertex set} and the \textit{edge set}, respectively. If $\{x, y\} \in E(G)$, then the vertices $x$ and $y$ are called \textit{adjacent}.  In what follows we will consider the simple graphs only. The {\it order} of a graph \(G\) is the cardinal number of the set \(V(G)\).
The graph \(G\) is called \emph{finite} if the order of \(G\) is finite. Let $v$ be a vertex of $G$. The cardinal number of the set 
$\{u \in V(G) : \{u, v\} \in E(G)\}$ 
is called the {\it degree} of the vertex $v$. The degree of $v$ will be denoted as $\operatorname{deg} v$.

\begin{definition}\label{776dcg}
Let $G_1$ and $G_2$ be graphs. A bijective mapping \linebreak
$\Phi \colon V(G_1)\to V(G_2)$
is called an {\it isomorphism} of $G_1$ and $G_2$ if the equalence 
\begin{equation*}
\{u,v\}\in E(G_1)\iff \{\Phi (u),\Phi (v)\}\in E(G_2)
\end{equation*}
is true for all $u,v\in V(G_1)$. 
The graphs are said to be  {\it isomorphic} if there is an isomorphism of these graphs.
\end{definition}

Let $G$ be a graph.
A graph \(G_1\) is a \emph{subgraph} of \(G\) if
\[
V(G_1) \subseteq V(G) \quad \text{and} \quad E(G_1) \subseteq E(G).
\]
In this case we  write \(G_1 \subseteq G\).

A \emph{path} is a finite graph \(P\) whose vertices can be numbered without repetitions so that
\begin{equation}\label{e3.3-1}
V(P) = \{v_1, \ldots, v_k\} \quad \text{and} \quad E(P) = \{\{v_1, v_2\}, \ldots, \{v_{k-1}, v_k\}\}
\end{equation}
with \(k \geqslant 2\). We write \(P = (v_1, \ldots, v_k)\)  if \(P\) is a path satisfying \eqref{e3.3-1} and said that \(P\) is a \emph{path joining \(v_1\) and \(v_k\)}. 
The {\it length} of the path $P$ is the number of edges that $P$ contains. 
A path $P_1 \subseteq G$ is called to be the {\it longest path} in $G$ if the inequality  
\begin{equation*}
|E(P)| \leq |E(P_1)|
\end{equation*}
holds for every path $P \subseteq G$.

A finite graph $C$ is a \textit{cycle} if $|V(C)|\geq 3$ and there exists an enumeration of its vertices without repetition such that 
\begin{equation*}
    V(C) = \{v_1, \ldots, v_n\}, \quad
E(C) = \{\{v_1, v_2\}, \ldots, \{v_{n-1}, v_n\}, \{v_n, v_1\}\}.
\end{equation*}

A graph \(G\) is \emph{connected} if for every two distinct vertices of \(G\) there is a path \(P \subseteq G\) joining these vertices.
 A connected graph without cycles is called a \textit{tree}.

By \textit{labeled tree} we understand a pair $(T, l)$, where $T$ is a tree and $l$ is a vertex labeling
$
l \colon V(T) \to \mathbb{R}^+.
$ We will use the symbol  $T(l)$ to denote the labeled tree $(T,l)$.

Let $ T(l)$ be a labeled tree and $d_l \colon V(T) \times V(T) \to \mathbb{R}^+$ be defined as
\begin{equation}
\label{e1.1}
d_l(u, v) :=
\begin{cases}
0, & \text{if } u = v, \\
\max\limits_{w \in V(P)} l(w), & \text{if } u\neq v,
\end{cases}
\end{equation}
where $P$ is the unique path joining $u$ and $v$ in $T$.

A labeling $l \colon V(T) \to \mathbb{R}^+$ is said to be {\it non-degenerate} if the inequality 
\begin{equation*}
    \max\{l(u),l(v)\}>0
\end{equation*}
holds for every $\{u,v\}\in E(T)$.

The following theorem was proved in \cite{Dov2020TaAoG}.

\begin{theorem}
\label{t1.4}
Let $T(l)$ be a labeled tree and let $d_l $ be deined by \eqref{e1.1}. Then $d_l$ is an ultrametric on $V(T)$ if and only if
$l \colon V(T) \to \mathbb{R}^+$ is non-degenerate.
\end{theorem}

Let $(X,d)$ be an ultrametric space.
In what follows we  say that $(X,d)$ is {\it generated by  labeled tree}  if  there exists a labeled tree $T(l)$ such that $X=V(T)$ and $(X,d) = (V(T), d_l)$. 

Let $L(T)$ be the set of all non-degenerate labelings defined on the vertex set of a fixed tree $T$. 
Then we use the symbol \( U(T) \) to denote the set of all ultrametric spaces \((V(T), d_l)\) with \(l \in L(T)\),
\[
U(T) := \{ (V(T), d_l) : l \in L(T) \}.
\]
If an ultrametric 
space $(X, d)$ belongs
to $U(T)$, then we say that
$(X, d)$ is {\it generated by vertex} labeling of $T$.

Let us recall  the concept of star graphs.

\begin{definition}
\label{wer}
A tree $S$ is  a \textit{star graph} if there exists a vertex $c \in V(S)$ such that $c$ is adjacent to every vertex of the set $V(S) \setminus \{c\}$, and no other edges exist.

\end{definition}

Let us denote by ${\bf US}$ the class of all ultrametric spaces generated by labeled star graphs.

The following result was proved in \cite{DR2025USGbLSG}.

\begin{theorem}\label{[2.1]}
Let $(X,d)$ be an ultrametric space. Then the following conditions are equivalent:
\begin{itemize}[left=10pt]
\item[(i)]  $(X,d) \in {\bf US}$.
\item[(ii)]  There is $x_{0} \in X$ such that the inequality $
        d(x_{0},x) \leq d(y,x) $
    holds whenever
$    
        x \neq y. 
        $
\end{itemize}
\end{theorem}

For the proof of the next result see Proposition 4.2 of \cite{DCR2025OJAC}.
\begin{theorem}
\label{8866gh}
Let (X,d) be an {\bf US}-space.
Then $(Y, d|_{Y \times Y})$ is also an ${\bf US}$-space for every finite non-empty $Y \subseteq X$.
\end{theorem}

\section{The main result and lemmas}

Let us start from the following lemma.

\begin{lemma}
    \label{rl}
Let \( T \) be a tree and let 
\begin{equation}
    \label{xa}
 U(T) \subseteq {\bf US}. 
 \end{equation}
Then the inequality
\begin{equation}
    \label{xa2}
|E(P)| \le 3 
\end{equation}
holds for every path \( P \subseteq T \).
\end{lemma}

\begin{proof}

Suppose that there is a path \( P_1 \subseteq T \) such that 
$
|E(P_1)| \geq 4.
$
Then there is also a path \( P_2 \subseteq P_1 \) having exactly four edges,
$
    |E(P_2)| = 4.
$
Let us define the labeling $l_2\colon V(P_2)\to {\mathbb R^+}$ as in Figure~\ref{fis1} below.

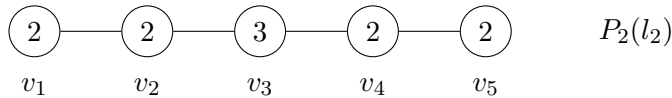
\begin{figure}[H]
\centering
\begin{center}
 \begin{tikzpicture}[scale=1, every node/.style={circle, draw, minimum size=5mm}]

 \node (F) at (0,0) {2};
 \node (G) at (1.5,0) {2};
 \node (H) at (3,0) {3};
 \node (I) at (4.5,0) {2};
 \node (J) at (6,0) {2};

 \draw (F) -- (G) -- (H) -- (I) -- (J);

 \node[draw=none, below=0mm of F] {$v_1$};
 \node[draw=none, below=0mm of G] {$v_2$};
 \node[draw=none, below=0mm of H] {$v_3$};
 \node[draw=none, below=0mm of I] {$v_4$};
 \node[draw=none, below=0mm of J] {$v_5$};

 \node[draw=none] at (8,0) {$P_2(l_2)$};
 \end{tikzpicture}
\end{center}
\caption{The vertex labeling $l_2$  of the path $P_2$.}
\label{fis1}
\end{figure}

It follows from Theorem~\ref{t1.4} that $d_{l_2} \colon V(P_2)\times V(P_2)\to\mathbb R^+$ is an ultrametric on $V(P_2)$, where $d_{l_2}$ defined by \eqref{e1.1} with $T=P_2$ and $l=l_2$. 
Moreover, using Theorem~\ref{[2.1]}, we can prove the relation 
\begin{equation}\label{xa3}
    (V(P_2),d_{l_2})\notin {\bf US}.
\end{equation}

It is easy to see that there exists a non-degenerate
\( l : V(T) \to \mathbb{R}^+ \)
such that the restriction of $l$ on the set $V(P_2)$ coincides with the labeling
\( l_2 : V(P_2) \to \mathbb{R}^+ \).
Inclusion \eqref{xa} implies the membership relation $(V(T), d_l) \in {\bf US}.$
Hence the ultrametric space $
(V(P_2), d_l \big|_{ V(P_2) \times V(P_2)})
$
also belongs to \( {\bf US}\) by Theorem~\ref{8866gh}. 
Using formula~\eqref{e1.1} it is easy to prove the equality
\[
d_{l_2} = d_l\big|_{V(P_2) \times V(P_2)}.
\]
Consequently we have
\[
(V(P_2), d_{l_2}) \in {\bf US}.
\]
The last statement contradicts relation~\eqref{xa3}. 

Inequality~\eqref{xa2} follows.
\end{proof}

Lemma \ref{rl} implies that the length of the longest path \( P^* \subseteq T \) does not exceed three if 
$
U(T) \subseteq {\bf US}.$
The inequality $|E(P^*)|\leq 3$ gives us an information about vertices $u,v\in V(T)$ whose degree is greater than one.

\begin{lemma}\label{wsdt}
Let \( T \) be a tree and let 
\begin{equation}
    \label{rfhi}
    |E(P)| \leq 3
\end{equation}
holds for every path $P\subseteq T$.  
Then 
every two distinct vertices \( u\) and \( v \) whose degree is greater than one, 
\begin{equation}
    \label{hhij}
\operatorname{deg} u \geq 2 \quad \text{and} \quad \operatorname{deg} v \geq 2,
\end{equation}
are adjacent,  
\begin{equation}
    \label{tyhj}
\{u,v\} \in E(T). 
\end{equation}
\end{lemma}

\begin{proof}
Let distinct $u,v\in V(T)$ satisfy \eqref{hhij}.
Suppose that
\begin{equation}
    \label{rvg}
    \{u,v\} \notin E(T). 
\end{equation}
Since every tree is a connected graph, relation~\eqref{rvg} implies the existence of a path \( P^0 \subseteq T \) such that
$
P^0= (v_1, \ldots, v_n),
$
with \( v_1 =u,\) \( v_n = v \) and \( n \geq 3 \).
Using inequalities \eqref{hhij} we can find some vertices $v_0, v_{n+1}$ of $T$ such that
\begin{equation}
    \label{ds1}
\{v_0, v_1\} \in E(T), \quad 
\{v_{n}, v_{n+1}\} \in E(T)
\end{equation}
and 
\begin{equation}
     \label{ds2}
     v_0 \neq v_2, \quad v_{n+1} \neq v_{n-1}.
\end{equation}
The tree $T$ is an acyclic graph. Consequently we also have
\begin{equation}     \label{ds3}
v_0 \notin \{v_2, \ldots, v_n\}
\quad {\rm and} \quad
     v_{n+1} \notin \{v_1, \ldots, v_{n-2}\}.
\end{equation}
Now relations \eqref{ds1}--\eqref{ds3} imply that
$
P^* = (v_0, v_1, \ldots, v_n, v_{n+1})
$
is a path joining $v_0$ and $v_{n+1}$ in $T$.
It is clear that $P^0\subseteq P^*\subseteq T$ and
\begin{equation}
    \label{ds4}
|E(P^*)| = |E(P^0)| + 2. 
\end{equation}
Furthermore, the inequality
\begin{equation*}
    |E(P^0)| \geq 2
\end{equation*}
holds by the definition of the path $P^0$. The last inequality and \eqref{ds4} give us the inequality
\[
|E(P^*)| \geq 4,
\]
that contradicts \eqref{rfhi}. 

Relation \eqref{tyhj} follows. 

\end{proof}

The next lemma almost directly follows from Lemma \ref{wsdt}.

\begin{lemma}\label{etyg}
Let \( T \) be a tree and let inequality \eqref{rfhi}
hold for every path $P\subseteq T$.  Then $T$ contains at most two vertices whose degree is greater than one.
\end{lemma}

\begin{proof}
    
Suppose that $u, v, w \in V(T)$ are distinct and that the inequality
$
\min\{\deg u, \deg v, \deg w\} \geq 2
$
holds. Then we have 
\begin{equation*}
\{\{u, v\}, \{v, w\},\{w, u\}\} \subseteq E(T)
\end{equation*}
by Lemma \ref{wsdt}.
Consequently $T$ contains the cycle $C$ with
\[
V(C) = \{u, v, w\}, \quad E(C) = \{\{u, v\}, \{v, w\}, \{w, u\}\},
\]
that is impossible by definition of trees. Thus $T$ has at most two vertices whose degree is greater than one, as was stated. 

\end{proof}

The following theorem is the main result of the present paper.

\begin{theorem}\label{sun}
Let $T$ be a tree.  
Then the following statements are equivalent:  

\begin{itemize}[left=10pt]
    \item[(i)] The inclusion  
    \begin{equation}
        \label{div1}
        U(T) \subseteq {\bf US}
       \end{equation}
    holds.   

    \item[(ii)] The inequality  
    \[
        |E(P)| \leq 3
    \]
    holds for every path $P \subseteq T$.  

    \item[(iii)] There exist at most two distinct vertices $u, v \in V(T)$ such that  
    \begin{equation}
        \label{fer1}
        \deg u \geq 2 \quad \text{and} \quad \deg v \geq 2.
        \end{equation}
\end{itemize}
\end{theorem}

\begin{proof}
The implication $(i) \Rightarrow (ii)$ is true by Lemma~\ref{rl}.  
Lemma~\ref{wsdt} and Lemma~\ref{etyg} imply the validity of the implication $(ii) \Rightarrow (iii)$.  
Thus we only need to prove that statement $(iii)$ implies statement $(i)$.

Let statement $(iii)$ be true. Suppose first that we have the inequality $\deg v \leq 1$ for every $v \in T$.  
Then, using Definition~\ref{wer}, we see that $T$ is a star graph, and, consequently~\eqref{div1} holds.  
Analogously, 
$T$ is a star graph and we have~\eqref{div1}
if there exists the unique vertex $v \in T$ with $\deg v \geq 2$.

Let us consider now the case
where \( T \) contains two distinct
vertices \( u \) and \( v \)
such that \eqref{fer1} holds. First of all we note that in this case statement $(iii)$ implies membership relation 
\begin{equation}
    \label{ji1}
    \{u,v\}\in E(T).
\end{equation}
Let \( L(T) \) be the set of all non-degenerate labelings
defined on \( V(T) \).
Inclusion \eqref{div1} holds if and only if the relation 
$(V(T),d_l)\in {\bf US}$ is true for each  $l \in L(T).$
Let \( l^*\in L(T) \) be arbitrary.
Then, without loss of generality, we
can assume that the inequality
\begin{equation}
    \label{div3}
l^*(u) \leq l^*(v) 
\end{equation}
holds.
By Theorem \ref{[2.1]} the membership relation
\[
(V(T), d_{l^*}) \in {\bf US}
\]
holds if we have
\begin{equation}
    \label{div4}
d_{l^*}(u, w) \leq d_{l^*}(z, w) 
\end{equation}
whenever \( z \in V(T) \) and \( z \neq w \).

Let us consider an arbitrary \( w \in V(T) \).
If \(  w=u \) holds, then \eqref{div4} follows
from the positivity property of ultrametric spaces.

Analogously, \eqref{div4} is trivially
valid if \( z = w \).

Suppose now that \( u, w \) and $z$
are pairwise distinct. 

Let us first consider the case
\begin{equation}
    \label{div6}
\{u,v\} \in E(T) \quad \text{and} \quad \{u,z\} \in E(T).
\end{equation}

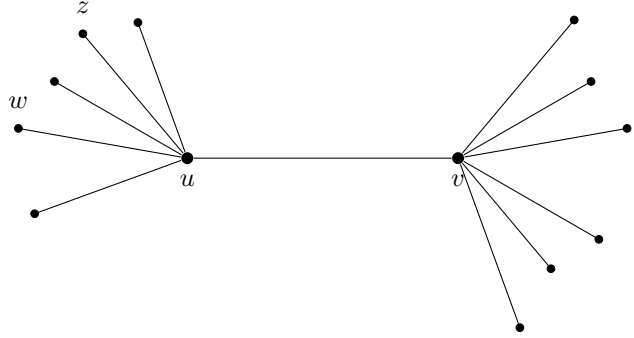
\begin{figure}[H]
\centering
\begin{center}
\begin{tikzpicture}[scale=1.2, every node/.style={font=\small}]
    \node[circle, fill=black, inner sep=1.6pt, label=below:$u$] (u) at (0,0) {};
    \node[circle, fill=black, inner sep=1.6pt, label=below:$v$] (v) at (3,0) {};

    \draw (u) -- (v);

    \foreach \ang/\len in {110/1.6,130/1.8,150/1.7,170/1.9,200/1.8} {
        \draw (u) -- ++(\ang:\len);
        \filldraw ($(u)+(\ang:\len)$) circle (1.2pt);
    }

    \coordinate (z) at ($(u)+(130:1.8)$);
    \coordinate (w) at ($(u)+(170:1.9)$);
    \node[label=above:$z$] at (z) {};
    \node[label=above:$w$] at (w) {};

    \foreach \ang/\len in {10/1.9,30/1.7,50/2.0,310/1.6,330/1.8,290/2.0} {
        \draw (v) -- ++(\ang:\len);
        \filldraw ($(v)+(\ang:\len)$) circle (1.2pt);
    }
\end{tikzpicture}

\end{center}
\caption{The vertices $w$ and $z$ are adjacent with the vertex $u$.}\label{far2}
\end{figure}

Relations \eqref{ji1} and \eqref{div6} imply that $(z,u,w)$ is a path joining $z$ and $w$ in $T$. (See Figure~\ref{far2}).
Now using~\eqref{e1.1} we obtain \eqref{div4},
\[
d_{l^*}(z, w) = \max\{ l^*(z), l^*(u), l^*(w) \}
\geq \max\{ l^*(u), l^*(w) \} = d_{l^*}(u, w).
\]

Let us consider the case
\begin{equation}
    \label{div7}
\{u,v\} \in E(T), \quad \text{but} \quad \{u,z\} \notin E(T).
\end{equation}
The last condition implies 
\begin{equation}
    \label{div8}
 \{u, w\} \in E(T) \quad \text{and} \quad \{v, z\} \in E(T).
\end{equation}

(See Figure \ref{far3} below.)

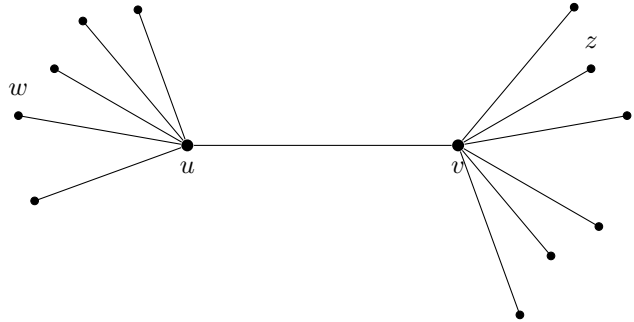
\begin{figure}[H]
\centering
\begin{center}

\begin{tikzpicture}[scale=1.2, every node/.style={font=\small}]
    \node[circle, fill=black, inner sep=1.6pt, label=below:$u$] (u) at (0,0) {};
    \node[circle, fill=black, inner sep=1.6pt, label=below:$v$] (v) at (3,0) {};

    \draw (u) -- (v);

    \foreach \ang/\len in {110/1.6,130/1.8,150/1.7,170/1.9,200/1.8} {
        \draw (u) -- ++(\ang:\len);
        \filldraw ($(u)+(\ang:\len)$) circle (1.2pt);
    }

    \coordinate (w) at ($(u)+(170:1.9)$);
    \node[label=above:$w$] at (w) {};

    \foreach \ang/\len in {10/1.9,30/1.7,50/2.0,310/1.6,330/1.8,290/2.0} {
        \draw (v) -- ++(\ang:\len);
        \filldraw ($(v)+(\ang:\len)$) circle (1.2pt);
    }

    \coordinate (z) at ($(v)+(30:1.7)$);
    \node[label=above:$z$] at (z) {};
\end{tikzpicture}

\end{center}
\caption{The vertex \( u \) is adjacent
with \( w \) but not adjacent with \( z \).}
\label{far3}
\end{figure}

Since $u$ and $w$ are adjacent by \eqref{div8}, equality \eqref{e1.1} gives us 
\begin{equation}
\label{kra1}
d_{l^*}(u,w)=\max \{l^*(u),l^*(w)\}.
\end{equation}
The path $(w,u,v,z)$ joins the vertices $w$ and $z$. Hence 
\begin{equation}
    \label{kra2}
d_{l^*}(z,w)=\max\{l^*(w),l^*(u),l^*(v),l^*(z)\}
\end{equation}
holds by \eqref{e1.1}.
Equalities \eqref{kra1} and \eqref{kra2} give us inequality \eqref{div4} for the case when \eqref{div7} holds.

To complete the proof of inequality \eqref{div4} if it sufficient to consider the case 
\begin{equation}
\label{kra3}
    \{u,w\}\notin E(T) \quad \text{and}\quad \{u,z\}\notin E(T).
\end{equation}
As in the case when \eqref{div7} holds we can rewrite \eqref{kra3} in the form
\begin{equation}
    \label{kra4}
    \{v,w\}\in E(T)\quad \text{and}\quad \{w,z\}\in E(T).
\end{equation}

(See Figure~\ref{far4} below).

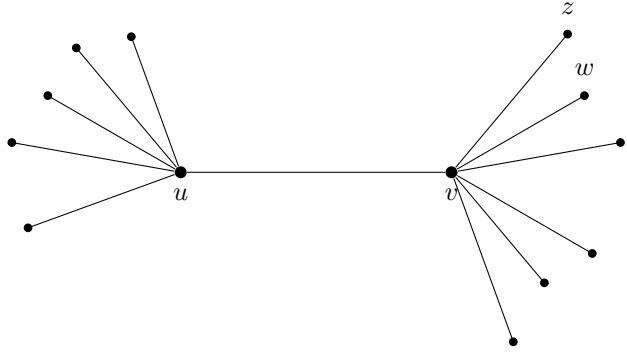
\begin{figure}[H]
\centering
\begin{center}
\begin{tikzpicture}[scale=1.2, every node/.style={font=\small}]
    \node[circle, fill=black, inner sep=1.6pt, label=below:$u$] (u) at (0,0) {};
    \node[circle, fill=black, inner sep=1.6pt, label=below:$v$] (v) at (3,0) {};
    \draw (u) -- (v);

    \foreach \ang/\len in {110/1.6,130/1.8,150/1.7,170/1.9,200/1.8} {
        \draw (u) -- ++(\ang:\len);
        \filldraw ($(u)+(\ang:\len)$) circle (1.2pt);
    }

    \foreach \ang/\len in {10/1.9,30/1.7,50/2.0,310/1.6,330/1.8,290/2.0} {
        \draw (v) -- ++(\ang:\len);
        \filldraw ($(v)+(\ang:\len)$) circle (1.2pt);
    }

    \coordinate (w) at ($(v)+(30:1.7)$);
    \node[label=above:$w$] at (w) {};

    \coordinate (z) at ($(v)+(50:2.0)$);
    \node[label=above:$z$] at (z) {};
\end{tikzpicture}
\end{center}
\caption{The vertices $w$ and $z$ are adjacent with vertex $v$.}
\label{far4}
\end{figure}

Since $(z,v,w) $ is a path in $T$, \eqref{kra4} and \eqref{e1.1} imply 
\begin{equation}
    \label{kra5}
    d_{l^*}(z,w)=\max\{l^*( z),l^*(v),l^*(w)\}.
\end{equation}

Analogously, we obtain
\begin{equation}
    \label{kra6}
    d_{l^*}(u,w)=\max\{l^*( u),l^*(v),l^*(w)\}
\end{equation}
because $(u, v, w)$ is a path in $T$.
Inequality \eqref{div4} follows from \eqref{div3}, \eqref{kra5} and \eqref{kra6}.

The proof is completed.

\end{proof}

Recall that a tree $T$ is called a \textit{double-star} graph if $T$ contains exactly two distinct vertices whose degrees are greater than one.

Theorems \ref{sun} implies the following corollary.

\begin{corollary}
    \label{xk}
Let $T$ be a tree.  
Then the inclusion $U(T) \subseteq {\bf US}$ holds if and only if $T$ is isomorphic either to a star-graph or to a double-star graph.
\end{corollary}

\begin{remark}
 Some results describing the properties of double-star graphs can be found in \cite{Akbari,Kart,Priy,Venk}.

\end{remark}

\section*{Declarations}

\subsection*{Declaration of competing interest}

 The authors declare no conflict of interest.

\subsection*{Data availability}

 All necessary data are included into the paper.

\subsection*{Funding}

First author was supported by grant $359772$ of the Academy of Finland.\\


\bibliographystyle{plainurl}
\bibliography{Longest}

\bigskip

CONTACT INFORMATION

\medskip
Oleksiy Dovgoshey\\
Department of Function Theory, Institute of Applied Mathematics and Mechanics of NASU, Slovyansk, Ukraine,\\
Department of Mathematics and Statistics, University of Turku, Turku, Finland \\
oleksiy.dovgoshey@gmail.com, oleksiy.dovgoshey@utu.fi

\medskip
Olga Rovenska\\
Department of Mathematics and Modelling, Donbas State Engineering Academy, Kramatorsk, Ukraine\\
rovenskaya.olga.math@gmail.com

\end{document}